\documentclass[11pt,letterpaper]{amsart}
\usepackage{amsmath} 
\usepackage[normalem]{ulem}
\usepackage{amsthm}
\usepackage{pinlabel}
\usepackage{amssymb} 
\usepackage{amsfonts}
\usepackage{graphicx} 
\usepackage{hyperref}
\usepackage{color,comment}
\usepackage{soul,xcolor,placeins}
\usepackage{ccaption}
\usepackage[shortlabels]{enumitem}
\usepackage{varwidth}
\usepackage{url}
\usepackage{multirow}
\usepackage{mathtools}
\usepackage[marginparsep=.1in,
            marginparwidth=1in,margin=1.25in]{geometry}
\makeatletter
\newcommand\makebig[2]{%
  \@xp\newcommand\@xp*\csname#1\endcsname{\bBigg@{#2}}%
  \@xp\newcommand\@xp*\csname#1l\endcsname{\@xp\mathopen\csname#1\endcsname}%
  \@xp\newcommand\@xp*\csname#1r\endcsname{\@xp\mathclose\csname#1\endcsname}%
}
\makeatother

\makebig{reallybig}{5}

\setlist[enumerate]{label=\normalfont{(\roman*)}}

\usepackage[OT2,OT1]{fontenc}
\newcommand\cyr
{
\renewcommand\rmdefault{wncyr}
\renewcommand\sfdefault{wncyss}
\renewcommand\encodingdefault{OT2}
\normalfont
\selectfont
}
\DeclareTextFontCommand{\textcyr}{\cyr}

\newcommand\congcat{\stackrel{\mathclap{\tiny\mbox{homeo/diff}}}{\quad\cong\quad}}

\definecolor{darkgreen}{rgb}{0,.5,0}
\newtheorem{theorem}{Theorem}[section]
\newtheorem{lemma}[theorem]{Lemma}
\newtheorem{proposition}[theorem]{Proposition}

\newtheorem{conjecture}[theorem]{Conjecture}

\theoremstyle{definition}

\newtheorem{definition}[theorem]{Definition}
\newtheorem{question}[theorem]{Question}

\newtheorem{remark}[theorem]{Remark}

\setstcolor{red}

\let\int\relax
\newcommand{\int}{\mathring}

\newcommand{\id}{\text{id}}

\title[Knotted handlebodies]{Knotted handlebodies in the 4-sphere and 5-ball}

    \author[Mark Hughes]{Mark Hughes}
    \address{Brigham Young University\\Provo, UT, 84602 USA}
    \email{hughes@mathematics.byu.edu}

    \author[Seungwon Kim]{Seungwon Kim}
\address{Sungkyunkwan University, Suwon, Gyeonggi-do, 16419 Republic of Korea}
\email{seungwon.kim@skku.edu}

    \author[Maggie Miller]{Maggie Miller}
    \address{Stanford University\\Stanford, CA, 94305 USA}
    \email{maggie.miller.math@gmail.com}
    

\thanks{MH was supported by a grant from the NSF (LEAPS-MPS-2213295). SK was supported by the Institute for Basic Science (IBS-R003-D1) and a National Research Foundation of Korea (NRF) grant funded by the Korea government (MSIT) (NRF-2022R1C1C2004559). MM was supported by a Clay Research Fellowship and a Stanford Science Fellowship.}

\let\oldtocsection=\tocsection
\let\oldtocsubsection=\tocsubsection
\let\oldtocsubsubsection=\tocsubsubsection
\renewcommand{\tocsection}[1]{\hspace{0em}\oldtocsection{#1}}
\renewcommand{\tocsubsection}[2]{\hspace{1em}\oldtocsubsection{#1}{#2}}
\renewcommand{\tocsubsubsection}[2]{\hspace{2em}\oldtocsubsubsection{#1}{#2}}

\begin{document}

\maketitle

\begin{abstract}
For every integer $g\ge 2$ we construct 3--dimensional genus--$g$ 1--handlebodies smoothly embedded in $S^4$ with the same boundary, and which are defined by the same cut systems of their boundary, yet which are not isotopic rel.\ boundary via any locally flat isotopy even when their interiors are pushed into $B^5$. This proves a conjecture of Budney--Gabai for genus at least 2.

\end{abstract}

\section{Introduction}\label{sec:intro}

In this paper, we work in both the smooth and topological locally flat categories. We will specify in which category various statements hold. As a shorthand, we will sometimes write ``topological", but implicitly mean ``topological and locally flat." 

The goal of this paper is to obstruct isotopies rel.\ boundary between two boundary-parallel handlebodies (by which we always mean 3-dimensional 1--handlebodies) that are properly embedded in $B^5$ and are homeomorphic rel.\ boundary as 3-manifolds.

\begin{definition}
Let $H_1$ and $H_2$ be genus--$g$ handlebodies that are both bounded by the same surface $F$. We say that $H_1$ and $H_2$ are {\emph{compressing curve equivalent}} if there exist $g$ disjoint simple closed curves $A_1,\ldots,A_g$ in $F$ such that $F\setminus\nu(A_i)$ is planar, and each $A_i$ bounds disks in both $H_1$ and $H_2$.
\end{definition}

If $H_1$ and $H_2$ are handlebodies properly embedded in $B^5$ with common boundary which are homeomorphic rel.\ boundary as 3-manifolds, then they are compressing curve equivalent.

Our motivation is the following conjecture of Budney and Gabai:

\begin{conjecture}[{\cite[Conjecture 11.3]{gabaibudney}}]\label{conjecture}
For each $g\ge 0$ there exist 3--dimensional genus--$g$ handlebodies $H_1,H_2\subset S^4$ such that $\partial H_1=\partial H_2$ and $H_1,H_2$ are compressing curve equivalent, but $H_1$ is not isotopic to $H_2$ via an isotopy that fixes $\partial H_i$.
\end{conjecture}

Budney and Gabai \cite{gabaibudney} provided examples satisfying Conjecture \ref{conjecture} for $g=0$, obstructing smooth isotopy rel.\ boundary. We prove a stronger version of this conjecture for $g\ge 2$.

\begin{theorem}\label{mainthm}There exist smooth genus--2 compressing-curve equivalent handlebodies $H_1$ and $H_2$ embedded in $S^4$ with $\partial H_1=\partial H_2$, such that if $H_1,H_2$ are boundary-summed with identical collections of $(g-2)$ smooth solid tori to obtain smooth genus--$g\ge2$ handlebodies $\widehat{H}_1,\widehat{H}_2$, the handlebodies $\widehat{H}_1$ and $\widehat{H}_2$ are not topologically isotopic rel.\ boundary even when their interiors are pushed into $B^5$.
\end{theorem}

In particular, in Theorem~\ref{mainthm}, boundary-summing $g-2$ solid tori to $H_1$ and $H_2$ yields a pair of genus-$g$ handlebodies satisfying Conjecture \ref{conjecture}.

In contrast, the 3-balls constructed by Budney--Gabai become smoothly isotopic rel.\ boundary when their interiors are pushed into $B^5$. This isotopy can be seen explicitly once one understands their construction, since Budney and Gabai construct their 3-balls explicitly.  In fact, any two 3-balls embedded in $S^4$ with the same boundary become isotopic rel.\ boundary when their interiors are pushed into $B^5$, as proved by Hartman \cite{hartman}. (This statement can be made in either the smooth or topological category.) This holds for pairs of $(n-1)$-dimensional balls embedded in $S^{n}$ for all $n\ge 3$; for disks in $S^3$ this follows easily from the Schoenflies theorem and in higher dimensions it follows from the unknotting conjecture.

\begin{proof}[Proof that $(n-1)$-balls in $S^n$ become isotopic in $B^{n+1}$ for $n\ge 4$]\leavevmode

Let $B_1$ and $B_2$ be $(n-1)$-balls embedded in $S^n$ with the same boundary. View $S^n$ as an equator of $S^{n+1}$, so that $S^n$ cuts $S^{n+1}$ into two balls $W$ and $W'$. Push the interior of $B_2$ slightly into $W$ so that $B_1\cup B_2$ is an embedded codimension-2 sphere inside $S^{n+1}=W\cup W'$.

The complement $S^{n+1}\setminus (B_1\cup B_2)$ is homotopy equivalent to a circle, so $B_1\cup B_2$ bounds an $n$-ball $V$ inside $S^{n+1}$ (by \cite{stallings} in the topological category; additionally \cite{levine} in the smooth category for $n>4$ or \cite[Corollary 3.1]{wall} and \cite[Theorem 2.1]{shaneson} in the smooth category for $n=4$).  If $V\subset W$, then $B_1$ (with interior pushed into $W$) is isotopic rel.\ boundary to $B_2$ in $W\cong B^{n+1}$ and we are done.

Suppose the interior of the ball $V$ intersects $W'$. Let $B_1\times I$ be a thickening of $B_1$ in $S^{n+1}$, so that \begin{itemize}
    \item $B_1$ is identified with $B_1\times\{1/2\}$,
    \item $B_1\times[0,1/2]\subset W'$ and $B_1\times [1/2,1]\subset W$,
    \item $\partial B_1\times[1/2,1]\subset B_2$.
    \end{itemize}
    Since $B_1$ is a ball, we can isotope $V$ rel.\ boundary so that $V\cap\nu(B_1)=B_1\times[1/2,1]\subset W$.

Note $\partial V\cap W'=B_1\subset\partial W'$. Then $W'$ is (homeo/diffeo)morphic to $B^{n-1}\times I\times I$ with $B_1=B^{n-1}\times \{1/2\}\times\{0\}$. (Please note that this parameterization is unrelated to the previous thickening of $B_1$.) Here, $B^{n-1}\times I\times\{0\}$ lies in $\partial W'$. Up to reparametrization, we have $\mathring{V}\cap W'$ contained in $B^{n-1}\times I\times[1/2,1]$, so we may isotope the interior of $V$ outside of $W'$ by isotopy along the second $I$ coordinate extended to be supported in a small neighborhood of $B^{n-1}\times I\times[1/2,1]\subset S^{n+1}$). Now $V$ is a ball cobounded by $B_1, B_2$ that lies completely within $W$.
\end{proof}

Our construction necessarily yields handlebodies of genus at least two. There is thus an obvious open question left about solid tori.

\begin{question}\label{question} Do there exist solid tori in $S^4$ with the same boundary that are compressing curve equivalent but are not isotopic rel.\ boundary? Do they necessarily become isotopic rel.\ boundary when their interiors are pushed into $B^5$?
\end{question}

Answering the first part of question \ref{question} positively would affirm Conjecture \ref{conjecture}. In a preprint of this paper, we also asked whether any two 3-balls in $S^4$ with the same boundary become isotopic rel.\ boundary when their interiors are pushed into $B^5$; this (as mentioned above) was answered positively by Hartman \cite{hartman}.

\subsection*{Acknowledgements} Thanks to Mark Powell for correcting our discussion on topological vs.\ smooth double slicing and to two anonymous referees for carefully reading the paper and providing many helpful comments.

\section{Double Slicing}

Our obstruction to isotopy rel.\ boundary comes from double sliceness (or more precisely, obstructing double sliceness) of 2-knots.

\begin{definition}
A {\emph{$2$--knot}} $K$ is the image of a smooth embedding from $S^{2}$ to $S^{4}$. 
We say that $K$ is (topologically/smoothly) {\emph{unknotted}} if $K$ is the boundary of the image of a (topological/smooth) embedding of $B^{3}$ in $S^{4}$.
\end{definition}

More generally, a positive-genus surface in $S^4$ is said to be (topologically/smoothly) unknotted if it bounds an embedded handlebody in $S^4$ in the appropriate category.

It is a theorem of Kervaire \cite{kervaire} that every (topological/smooth) 2--knot is {\emph{slice}}, in the sense that it bounds a (topological/smooth) 3--ball in $B^5$. However, Stoltzfus \cite{stoltzfus} showed that not every 2--knot is topologically {\emph{doubly slice}}.

\begin{definition}
Let $K$ be a $2$--knot. We say that $K$ is (topologically/smoothly) {\emph{doubly slice}} if, writing $S^{5}$ as the union of two $5$--balls along their boundary $W\cong S^4$, 
there exists a (topological/smooth) embedding $f:B^4\to S^5$ such that \[(W,W\cap f(\partial B^4))\congcat(S^4,K).\]

In words, $K$ is doubly slice when $K$ is an equator of an unknotted 3--sphere in $S^5$ in the appropriate category.
\end{definition}

Ruberman \cite{ruberman} gave convenient examples of 2--knots that are not doubly slice (using different techniques than Stoltzfus, who actually obstructed algebraic double sliceness, a related property that is implied by double sliceness).

\begin{theorem}[\cite{ruberman}]\label{ruberman}
The 5-twist spun trefoil is not smoothly doubly slice.
\end{theorem}

While Stoltzfus obstructs topological double sliceness (i.e.\ obstructs a 2-knot from being a cross-section of a locally flat, topologically unknotted 3-sphere), Ruberman's theorem involves smooth topology. Ruberman gives an invariant that obstructs double sliceness which is shown to be well-defined using Rokhlin's theorem applied to a smooth, spin 4-manifold cobounded by a smooth 3-manifold in $B^5$. When applied directly, he thus obstructs the 5-twist spun trefoil from being smoothly doubly slice. By work of Wall \cite[Corollary 3.1]{wall} and Shaneson \cite[Theorem 2.1]{shaneson} (or more precisely a theorem of Wall that rested on a conjecture later proved by Shaneson), every smooth 3-sphere in $S^5$ that is topologically unknotted is also smoothly unknotted. Thus, we can rephrase Theorem \ref{ruberman} in a seemingly sharper way: if $L$ is a smooth 3-sphere in $S^5$ admitting the 5-twist spun trefoil as a cross-section (via a smooth splitting of $S^5$), then $L$ is not topologically unknotted.

This is a subtle point -- Hillman \cite{hillmansolvable} showed that the 5-twist spun trefoil is a cross-section of a locally flat unknotted 3-sphere, i.e. is topologically doubly slice. We conclude that such a 3-sphere cannot be smoothed without changing its intersection with the 4-sphere.

We focus on Ruberman's obstruction rather than Stoltzfus's because it is easier for us to give an explicit example of a 2-knot to which Ruberman's proof applies. This is important because we will use another property of this particular 2-knot which we discuss in the next section (see Proposition \ref{satoh}).

\section{Constructing slice 3--balls}

By Kervaire \cite{kervaire}, we know that every 2--knot is slice. In this section we give a procedure for constructing a 3--ball in $B^5$ bounded by a specific 2--knot in $\partial B^5$.

\begin{definition}
Let $\Sigma$ be an oriented genus--$g$ surface in an orientable 4-manifold $X$. Let $\eta$ be an arc in $X$ with endpoints on $\Sigma$ that is disjoint from $\Sigma$ in its interior and is not tangent to $\Sigma$ near its boundary.  Let $h$ be a 3-dimensional 1-handle with core arc $\eta$ and feet on $\Sigma$ with the property that surgering $\Sigma$ along $H$ yields an orientable genus-$(g+1)$ surface $\Sigma_{\eta}$. By Boyle \cite{boyle}, the handle $h$ is determined by $\eta$ up to smooth isotopy in a neighborhood of $\eta$. 

We say that $\Sigma_\eta$ is obtained from $\Sigma$ by {\emph{attaching a tube along $\eta$}}.
\end{definition}

The following lemma of Hosokawa--Kawauchi \cite{hosokawa} is very well known (and has been proved in much greater generality by Baykur--Sunukjian \cite{baykur2016knotted}).

\begin{lemma}[\cite{hosokawa}]
\label{obs:stabilization}
Let $K$ be a smooth 2-sphere in $S^4$. For some $n$, there exists a collection of $n$ arcs $\eta_1,\ldots, \eta_n$ such that attaching smooth tubes to $K$ along $\eta_1,\ldots,\eta_n$ yields a smoothly unknotted genus--$n$ surface.
\end{lemma}

\begin{proof}
Let $Y$ be an oriented 3-manifold smoothly embedded in $S^4$ with boundary $K$. Fix a relative handle decomposition on $Y$. Let $\eta_1,\ldots,\eta_n$ be cores of the 1--handles of this decomposition. Then $K_{\eta_1,\ldots,\eta_n}$ bounds a copy of $Y$ with the relative 1--handles deleted, which is a smooth handlebody. We conclude that $K_{\eta_1,\ldots,\eta_n}$ is smoothly unknotted.
\end{proof}

Satoh \cite{satoh} gave examples of when the tubings prescribed by Lemma \ref{obs:stabilization} are particularly simple.

\begin{proposition}[\cite{satoh}]\label{satoh}
Let $K$ be a $k$--twist spun trefoil for some $k$. Then a single tube can be attached to $K$ to obtain a smoothly unknotted torus.
\end{proposition}

Before stating the main lemma of this section, we describe some useful work of Hirose on isotopies of unknotted surfaces that makes use of the Rokhlin quadratic form.

\begin{definition}
Let $\Sigma$ be a genus--$g$ surface in $S^4$. The {\emph{Rokhlin quadratic form}} on $\Sigma$ is a quadratic form $q:H_1(\Sigma;\mathbb{Z})\to\mathbb{Z}/2\mathbb{Z}$ defined as follows.

Given a primitive element $\alpha\in H_1(\Sigma;\mathbb{Z})$, let $C$ be a simple closed curve on $\Sigma$ representing $\alpha$. Let $P$ be a disk in $S^4$ bounded by $C$ that is {\emph{framed}}, i.e.\ so that the 1--dimensional subbundle of the normal bundle of $C$ that is tangent to $\Sigma$ extends over all of $P$. Then \[q(\alpha)=|\mathring{P}\cap \Sigma|\pmod{2}.\]
\end{definition}

For our purposes, a {\emph{symplectic basis}} $((A_1,B_1)$, $\ldots$, $(A_g,B_g))$ of a genus--$g$ surface $F$ consists of simple closed curves $A_1$, $\ldots$, $A_g$, $B_1$, $\ldots$, $B_g$ on $F$ such that the following are all true.
\begin{itemize}
    \item $[A_1],\ldots,[A_g]$ are linearly independent in $H_1(F;\mathbb{Z})$, 
    \item $[B_1],\ldots,[B_g]$ are linearly independent in $H_1(F;\mathbb{Z})$,
    \item $A_i\cap A_j=B_i\cap B_j=A_i\cap B_j=\emptyset$ for $i\neq j$,
    \item $A_i$ and $B_i$ intersect transversely in one point.
\end{itemize}

Hirose \cite{hirose2002diffeomorphisms} showed that the Rokhlin form determines equivalence of symplectic bases on unknotted surfaces in $S^4$.

\begin{theorem}[\cite{hirose2002diffeomorphisms}]
\label{thm:hirose}
Let $U_g$ be an unknotted surface of genus--$g$ in $S^4$. Fix two symplectic bases of curves $((A_1,B_1),\allowbreak\ldots,\allowbreak(A_g,B_g))$ and $((A'_1,B'_1),\allowbreak\ldots,\allowbreak(A'_g,B'_g))$ on $U_g$. Then there is an ambient isotopy of $S^4$ taking $U_g$ to itself and taking $A_i,B_i$ to $A'_i,B'_i$ for each $i$ if and only if $q([A_i])=q([A'_i])$ and $q([B_i])=q([B'_i])$ for each $i$.
\end{theorem}

\begin{lemma}
\label{lem:ball}
Let $h:B^5\to[0,1]$ be the radial function. If a 2--knot $K$ can be transformed into an unknotted surface $U_n$ by attaching $n$ tubes, then $K$ bounds a 3--ball $B$ in $B^5$ such that $h|_{B}$ is Morse with one index--0 point, $n$ index--1 points, and $n$ index--2 points.

\end{lemma}

\begin{proof}
Let $A_1,\ldots, A_n$ be belt circles of the tubes attached to $K$ to obtain $U_n$. Since each $A_i$ bounds a framed disk (the cocore of the 3--dimensional 1--handle $H_i$ used to perform the tube surgery) whose interior is disjoint from $U_n$, $q([A_i])=0$. Choose curves $B_1,\ldots, B_n$ on $U_n$ such that $((A_1,B_1),\allowbreak\ldots,\allowbreak (A_n,B_n))$ is a symplectic basis of $U_n$. 

If $q([B_i])=1$, then let $C_i$ be a curve obtained by cut-and-pasting $A_i$ and $B_i$, so $[C_i]=[A_i]+[B_i]$ and the curves $A_i, B_i, C_i$ pairwise intersect in a single point. Since $q$ is a quadratic form, we have $q([C_i])=q([A_i])+q([B_i])+|A_i\cap B_i|=0+1+1=0\in\mathbb{Z}/2\mathbb{Z}$. Then redefine $B_i:=C_i$; we thus arrange for $q([B_i])=0$ for all $i$.

By Theorem \ref{thm:hirose}, $U_n$ can be isotoped such that $((A_1,B_1),\ldots, (A_n,B_n))$ is taken to the standard symplectic basis (see Figure~\ref{fig:standard}), so we conclude that $B_1,\ldots, B_n$ bound disjoint framed disks $\Delta_1,\ldots,\Delta_n$ whose interiors are in the complement of $U_n$. Specifically, the disks $\Delta_1,\ldots,\Delta_n$ may be taken to lie in a copy of $S^3$ that contains $U_n$. These disks have the property that when $\Delta_i$ is thickened to $\Delta_i\times I$ (so that $(\partial\Delta_i)\times I$ is contained in $U_n$ and $\mathring{\Delta}_i\times I$ is disjoint from $U_n$), compressing $U_n$ along all of the $\Delta_i$ yields the unknotted sphere $U_0$, which bounds a 3-ball $D$. We can now describe $B$ via the following intersections. (Recall that $h^{-1}(1)=\partial B^5$ and that $h^{-1}(0)$ is the central point of $B^5$.)
\begin{figure}
    \centering
    \includegraphics[width=0.8\textwidth]{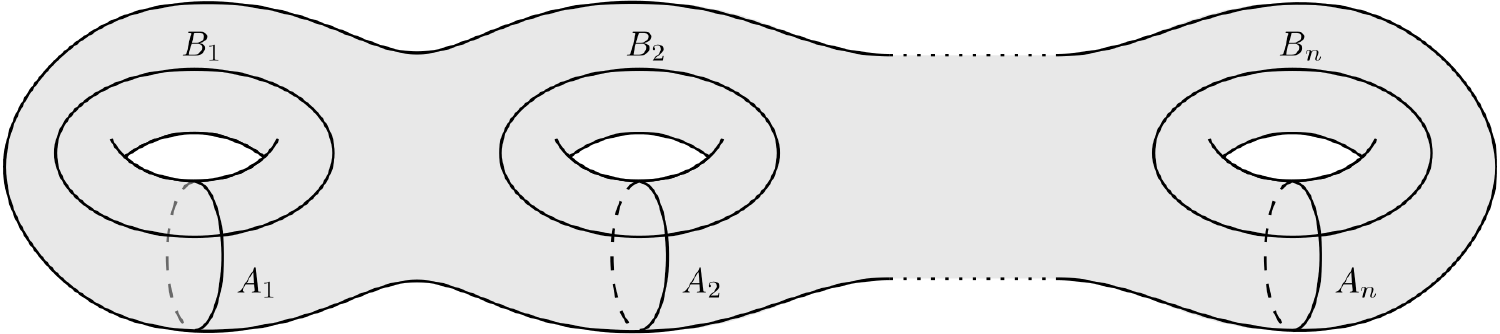}
    \caption{A symplectic basis $((A_1,B_1),\ldots, (A_n,B_n))$ on an unknotted genus--$n$ surface in $S^4$. We have shaded a genus--$n$ handlebody in which the $A_i$ curves bound disks; the closure of its complement in this 3-dimensional cross-section is a genus--$n$ handlebody in which the $B_i$ curves bound disks. Gluing these two handlebodies together yields an $S^3$ that splits $S^4$ into two smooth 4-balls.}
    \label{fig:standard}
\end{figure}
\begin{align*}
    B\cap h^{-1}(3/4,1]&=K\times(3/4,1],\\
    B\cap h^{-1}\{3/4\}&=K\cup \left(\bigcup_{i=1}^n H_i\right),\\
    B\cap h^{-1}(1/2,3/4)&=U_n\times(1/2,3/4),\\
    B\cap h^{-1}\{1/2\}&=U_n\cup\left(\bigcup_{i=1}^n (\Delta_i\times I)\right),\\
    B\cap h^{-1}(1/4,1/2)&=U_0\times(1/4,1/2),\\
    B\cap h^{-1}\{1/4\}&=D,\\
    B\cap h^{-1}[0,1/4)&=\emptyset.
\end{align*}

In words, $B$ is built from $\partial B=K\times\{1\}$ by the following steps (in order).
\begin{enumerate}[label=\arabic*.]
    \item Thicken $K$.
    \item Attach $n$ 3-dimensional 1-handles whose belts are $A_1,\ldots, A_n$.
    \item Attach $n$ 3-dimensional 2-handles along curves $B_1,\ldots, B_n$ that are chosen so that $|A_i\cap B_j|=\delta_{ij}$.
    \item Attach a 3-dimensional 3-handle to the boundary component which is not $K$.
\end{enumerate}

Because $|A_i\cap B_j|=\delta_{ij}$, the 1- and 2-handles in this decomposition of $B$ can be canceled, and hence $B$ is a 3-ball.%

After a small perturbation of $B$, $h|_B$ is Morse with one index--0 critical point (in $h^{-1}(1/4)$), $n$ index--1 critical points (in $h^{-1}(1/2)$) and $n$ index--2 critical points (in $h^{-1}(3/4)$).
\end{proof}

\section{Proof of Theorem~\ref{mainthm}}

The ability to position the handlebody $H$ in $B^5$ so that $h|_{H}$ has only one index-0 point will be particularly useful.

\begin{lemma}
\label{lem:boundaryparallel}
Let $H$ be a genus--$g$ handlebody smoothly and properly embedded in $B^5$, and let $h:B^5\to[0,1]$ be the radial function.  Assume that the function $h|_H$ is Morse with a single index--0 critical point and with no index--2 or 3 critical points.  Then there is a smooth isotopy of $H$ rel.\ boundary taking $H$ into $\partial B^5$.
\end{lemma}

\begin{proof}
After choosing a gradient-like flow for $h|_H$, $h$ induces a handlebody decomposition of $H$ with one 0-handle and $g$ $1$-handles. Let $t_0$ and $t_1$ be chosen so that $0<t_0<t_1<1$, with the index--0 critical point of $h|_H$ lying below $h^{-1}(t_0)$, and the $g$ index--1 critical points sitting between $h^{-1}(t_0)$ and $h^{-1}(t_1)$.  Then in $h^{-1}(t_0) \cong S^4$ the level set $S:=h|_H^{-1}(t_0)$ is an unknotted 2--sphere, which bounds a properly embedded 3--ball $W=h|_H^{-1}[0,t_0]$ in $h^{-1}([0,t_0])$. Let $W'$ be the image of $W$ after an isotopy rel.\ boundary to $h^{-1}(t_0)$, so $W'$ is a 3-ball in $h^{-1}(t_0)$ bounded by $S$.

As $t$ increases from $t_0$ to $t_1$, the cross-sections $h|_H^{-1}(t)$ of $H$ change by attaching tubes along some arcs $\eta_1, \ldots , \eta_g$. Push these tubes down to $h^{-1}(t_0)$, so that $h|_H^{-1}(t_0)$ consists of the union of the unknotted 2--sphere $S$ along with $g$ 3--dimensional 1--handles $b_1,\ldots, b_g$ attached to $S$ along each $\eta_1,\ldots,\eta_g$ respectively.  For small $\varepsilon > 0$, nearby level sets $h|_H^{-1}(t_0-\varepsilon)$ now consist of only (a parallel copy of) the sphere $S$, while $h|_H^{-1}(t_0+\varepsilon)$ is a genus--$g$ surface parallel to one obtained from adding tubes to $S$ along the arcs $\eta_i$.  

Because $\pi_1(S^4\setminus S)\cong\mathbb{Z}$, any two arcs based at a pair of  points in $S$ and with interiors disjoint from $S$ are homotopic and hence isotopic in $S^4\setminus S$. This allows us to isotope $H$ so that the arcs $\eta_i$ (and hence the 3-dimensional 1-handles $b_i$) avoid the ball $W'$ in $h^{-1}(t_0)$.

Now we can isotope $W$ to $W'\subset h^{-1}(t_0)$, so $M:=h_H^{-1}(t_0)$ is the genus-$g$ handlebody $W'\cup b_1\cup\cdots\cup b_g$. If we push the interior of $M$ slightly below $h^{-1}(t_0)$, the function $h|_H$ has no critical values in $[t_0,1]$.  The level sets $h|_H^{-1}(t)$ for $t_0 \leq t \leq 1$ trace out an isotopy of a genus-$g$ surface $F$ in $S^4$ that can be extended to an ambient isotopy $f_t$ ($t_0\le t\le 1$) of $S^4$, with $f_{t_0}=\id$. If we parametrize $h^{-1}([t_0,1]) \cong \partial B^4 \times [t_0,1]$, then 
\[
H=\{(f_{t_0}(M),t_0)\}\cup \{(f_t(F),t)\, | \, t_0\le t\le 1\},
\] 
which is isotopic rel.\ boundary to $\{(f_{1}(M),t_0)\}\cup \{(f_1(F),t)\, | \, t_0\le t\le 1\}$.  This can in turn be pushed into $\partial B^5$.
\end{proof} 

We are now ready to prove Theorem \ref{mainthm}.

\begin{proof}[Proof of Theorem \ref{mainthm}]

First we prove the theorem for $g=2$. We will then extend this strategy to larger $g$. 
Let $K$ be the 5-twist spun trefoil. By Theorem \ref{ruberman}, $K$ is not smoothly doubly slice. By Proposition~\ref{satoh} and Lemma~\ref{lem:ball}, there is a smoothly embedded 3-ball $B$ in $B^5$ whose boundary is $K$ and such that the radial function on $B^5$ restricts to a Morse function on $B$ with one index--0 point, one index--1 point, and one index--2 point.

Double $B$ along $K$ to obtain a smooth 3--sphere $L$ in $S^5$. (That is, $(S^5,L)=(B^5,B)\cup\overline{(B^5,B)}$. We will write $B$ and $\overline{B}$ to denote the corresponding halves of $L$.)  By replacing the radial function $h$ on $\overline{(B^5,B)}$ with $2-h$, and gluing to the radial function on $(B^5,B)$, we obtain a function $S^5 \rightarrow [0,2]$ which (by abuse of notation) we continue to denote by $h$.  This new function restricts to a Morse function $h|_L$ on $L$ with the following critical points, in order from highest to lowest (descending in the table, naturally).
\begin{center}
\begin{tabular}{clp{.1in}l}(vi)&index--3&&from $\overline{B}$; the dual of the index--0 point of $B$,\\(v)&index--2&&from $\overline{B}$; the dual of the index--1 point of $B$,\\(iv)&index--1&&from $\overline{B}$; the dual of the index--2 point of $B$,\\(iii)&index--2&&from $B$,\\(ii)&index--1&&from $B$,\\(i)&index--0&&from $B$.\end{tabular}
\end{center}

Note that the critical points of $h|_L$ are not in order. However, we may interchange the heights (with respect to $h$) of the (iv) index--1 point and the (iii) index--2 point by smoothly isotoping $L$, so that both of the index--1 points of $h|_L$ are below both the index--2 points.  After this isotopy, fix a level $S^4\cong h^{-1}(t_0)$ between the index--1 and index--2 critical points of $h|_L$ separating $S^5$ into two 5-balls $V_1:=h^{-1}[0,t_0]$ and $V_2:=h^{-1}[t_0,2]$.  This $S^4$ intersects $L$ in a smooth genus--2 unknotted surface $U=h|_L^{-1}(t_0)$. We have $L=H_1\cup_{U}H_2$ for two smooth genus--2 handlebodies $H_1$ and $H_2$, with $H_1\subset V_1$ lying below $U$ and $H_2\subset V_2$ lying above $U$. Note that $h|_{H_1}$ and $-h|_{H_2}$ each have one index-0 point and two index-1 points. By Lemma \ref{lem:boundaryparallel}, $H_1$ is smoothly boundary-parallel in $V_1$ and $H_2$ is smoothly boundary-parallel in $V_2$.

Since $(H_1,H_2)$ is a Heegaard splitting of $S^3$, by Waldhausen's theorem (\cite{waldhausen}; see \cite{schleimer} for exposition) there exists a symplectic basis $((A_1,B_1),\allowbreak(A_2,B_2))$ of $U$ such that each $A_i$ bounds a disk in $H_1$ and each $B_i$ bounds a disk in $H_2$. Since $H_1,H_2$ are each boundary parallel, we may isotope such a disk bounded by $A_i$ or $B_i$ from $B^5$ to $S^4$ to obtain a framed disk with boundary on $U$ and interior disjoint from $U$. We conclude $q([A_i])=q([B_i])=0$.

By Theorem \ref{thm:hirose}, there is a diffeomorphism of $S^4$ taking $(U;(A_1,B_1),(A_2,B_2))$ to the standard unknotted surface with standard curves as in Figure \ref{fig:standard} (drawn for general genus). Then $U$ bounds smooth handlebodies $H_1^*$ and $H_2^*$ in $S^4$ with $A_i$ bounding a disk in $H_1^*$ and $B_i$ bounding a disk in $H_2^*$, and with $H_1^*\cup_{U}H_2^*$ an unknotted 3-sphere. Push $H_1^*$ into $V_1$ and $H_2^*$ into $V_2$. Since $L=H_1\cup_U H_2$ is topologically knotted, $L$ is not topologically isotopic to $H_1^*\cup_UH_2^*$, which is unknotted. Therefore, if $H_1$ is topologically isotopic rel.\ boundary in $V_1\cong B^5$ to $H_1^*$, then $H_2$ is not topologically isotopic rel.\ boundary in $V_2\cong B^5$ to $H_2^*$. This completes the proof for $g=2$, with the pair of non-isotopic handlebodies being either $(H_1,H_1^*)$ or $(H_2,H_2^*)$.

To extend the above argument to larger $g$, we simply perturb the 3-ball $B$. Fix $g>2$ and let $\widehat{B}$ be obtained from $B$ by perturbing $B$ with respect to $h$ to introduce $(g-2)$ pairs of cancelling index-1, index-2 pairs to $h|_{\widehat{B}}$. Now consider the 3-sphere $\widehat{L}=\widehat{B}\cup\overline{B}$ in $S^5$. Again, $K$ is a cross-section of $\widehat{L}$, so $\widehat{L}$ is not topologically unknotted. (And more directly, $\widehat{L}$ is smoothly isotopic to $L$ so of course $\widehat{L}$ is not topologically unknotted.) In words, we obtain $\widehat{L}$ by gluing a copy of $\widehat{B}$ to a copy of $B$ (with opposite orientations); note that $\widehat{L}$ is not expressed as a double. As constructed, the radial function on $B^5$ restricts to a Morse function on $\widehat{L}$ with the following critical points.

\begin{center}
\begin{tabular}{clp{.1in}l}(vii)&index--3&&from $\overline{B}$; the dual of the index--0 point of $B$,\\(vi)&index--2&&from $\overline{B}$; the dual of the index--1 point of $B$,\\(v)&index--1&&from $\overline{B}$; the dual of the index--2 point of $B$,\\(iv$_{2(n-2)}$)&index--2&&\multirow{5}{*}{$\Bigg\}$From the perturbations that yield $\widehat{B}$ from $B$,}\\(iv$_{2(n-2)-1}$)&index-1&&\\$\vdots$&\multicolumn{1}{c}{$\vdots$}&&\\(iv$_{2}$)&index--2&&\\(iv$_{1}$)&index-1&&\\(iii)&index--2&&from $B$,\\(ii)&index--1&&from $B$,\\(i)&index--0&&from $B$.\end{tabular}
\end{center}

In total, $h|_{\widehat{L}}$ has one index--0 point, $g$ index--1 points, $g$ index--2 points, and one index--3 point.  
Smoothly isotope $\widehat{L}$ to move the index-1 critical points below $h^{-1}(t_0)\cong S^4$ and the index-2 critical points above $h^{-1}(t_0)$. Then $h^{-1}(t_0)$ intersects $\widehat{L}$ in a genus--$g$ surface $\widehat{U}$, and $\widehat{H}_1=h|_L^{-1}[0,t_0]$ and $H_2=h|_L^{-1}[t_0,1]$ are smooth genus--$g$ handlebodies that are smoothly boundary parallel (via Lemma \ref{lem:boundaryparallel}) in the 5-balls $V_1,V_2$ respectively. By the same argument as in the $g=2$ case (recall Figure \ref{fig:standard}), $\widehat{U}$ bounds smooth boundary-parallel handlebodies $\widehat{H}_1^*$ and $\widehat{H}_2^*$ respectively in $V_1,V_2$ such that $\widehat{H}_i$ and $\widehat{H}_i^*$ are compressing-curve equivalent but $\widehat{H}_1^*\cup_{\widehat{U}}\widehat{H}_2^*$ is unknotted. We similarly conclude that if $\widehat{H}_1$ is topologically isotopic rel.\ boundary to $\widehat{H}_1^*$, then $\widehat{H}_2$ is not topologically isotopic rel.\ boundary to $\widehat{H}_2^*$. Then either $(\widehat{H}_1,\widehat{H}_1^*)$ or $(\widehat{H}_2,\widehat{H}_2^*)$ are the desired pair of non-isotopic genus-$g$ handlebodies. 
\end{proof}

\begin{remark}\label{fixh1remark}
While not strictly necessary in the proof of Theorem \ref{mainthm}, we can modify the argument slightly so that the non-isotopic pair of handlebodies is specified (rather than being indeterminately one of $(H_1,H_1^*)$ or $(H_2,H_2^*)$).

To accomplish this, return to the genus-2 case and recall that $H_1\cup H_2$ is the knotted 3-sphere $L\subset S^5$, with $H_i\subset V_i\cong B^5$ so that $L$ intersects $S^4=\partial V_i$ in an unknotted genus-2 surface $U$. Let $((A_1,B_1),(A_2,B_2))$ be a symplectic basis of $U$ with each $A_j$ bounding a disk in $H_1$ and each $B_j$ bounding a disk in $H_2$. Push $H_1,H_2$ into $S^4$. Perform smooth isotopy of $S^4$ (extended to all of $S^5$) that takes $H_1$ to a handlebody in a smooth equatorial $S^3$ of $S^4$, and let $H_3:=\overline{S^3\setminus H_1}$. Set $H_1^*:=H_1$. Note that this isotopy need not fix $U$, and will take $H_2$ to some potentially complicated handlebody in $S^4$ with the same boundary as $H_1$. If $H_1,H_2$ are pushed back into $V_1,V_2$ respectively, their union is a 3-sphere isotopic to $L$, so still not topologically isotopic to the unknotted 3-sphere.

If $H_2,H_3$ are compressing curve equivalent, then we are done: set $H_2^*:=H_3$ and push the interior of each $H_i$ and $H_i^*$ slightly into $V_i$. Since $H_1^*\cup H_2^*$ is an unknotted $S^3$ and $H_1$,$H_1^*$ are isotopic rel.\ boundary in $V_1$, the handlebodies $H_2,H_2^*$ are not topologically isotopic rel.\ boundary in the 5-ball $V_2$.

In general, we cannot expect for $H_2,H_3$ to be compressing curve equivalent. Let $C_1,C_2$ be curves on $U$ bounding disks in $H_3$ so that $((A_1,C_1),(A_2,C_2))$ are a symplectic basis for $U$ (again using Waldhausen's theorem). Take the intersection points $A_i\cap B_i$ and $A_i\cap C_i$ to agree for each $i$, and let $\phi:U\to U$ be a surface automorphism with $\phi(B_j)=C_j$ for each $j$ and that fixes each $A_i$ pointwise. Then $\phi$ restricts to an boundary-fixing automorphism of the planar surface $F:=U\setminus\nu(A_1\cup A_2)$. Let $\gamma_1,\gamma_2$ be separating curves on $F$ as in Figure \ref{fig:abccurves}. Then $\pi_0($Aut$(F))$ is generated by Dehn twists about the four boundary components of $F$ and the curves $\gamma_1,\gamma_2$. In particular, this means that up to isotopy, $\phi(\gamma_1)$ is obtained from $\gamma_1$ by a sequence of Dehn twists about $\gamma_1$ and $\gamma_2$.

\begin{figure}
    \centering
    \labellist
    \pinlabel{\textcolor{red}{$A_1$}} at 47 -7
    \pinlabel{\textcolor{red}{$A_2$}} at 108 -7
    \pinlabel{\textcolor{red}{$\gamma_1$}} at 78 -7
    \pinlabel{\textcolor{red}{$\gamma_2$}} at 78 32
    \pinlabel{\textcolor{blue}{$B_1$}} at 7 25
    \pinlabel{\textcolor{blue}{$B_2$}} at 145 58
    \pinlabel{\textcolor{darkgreen}{$C_1$}} at 23 55
    \pinlabel{\textcolor{darkgreen}{$C_2$}} at 130 20
    \endlabellist
    \includegraphics[width=0.5\textwidth]{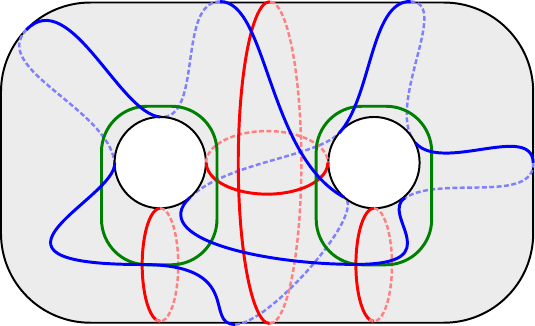}
    \vspace{.1in}
    \caption{The surface $U$, on which $((A_1,B_1),(A_2,B_2))$ is a symplectic basis, as is $((A_1,C_1),(A_2,C_2))$. For each $i$, the curve $A_i$ bounds a disk into $H_1$, the curve $B_i$ bounds a disk into $H_2$, and $C_i$ bounds a disk into $H_3$ (see Remark \ref{fixh1remark}). We include curves $\gamma_1,\gamma_2$. There is an automorphism $\phi$ of $U$, fixing $A_1$ and $A_2$ pointwise, that takes $B_i$ to $C_i$. Up to isotopy rel.\ boundary in the complement $F:=U\setminus\nu(A_1\sqcup A_2)$, the map $\phi$ is a product of Dehn twists about $\gamma_1$, $\gamma_2$, and curves parallel to components of $\partial F$. Here we draw a general situation, but in Remark \ref{fixh1remark} we show how to perform an isotopy of $S^5$ before choosing $C_1,C_2$ so that $C_i=B_i$ for each $i$, and thus $H_2$ and $H_3$ are compressing curve equivalent.}
    \label{fig:abccurves}
\end{figure}
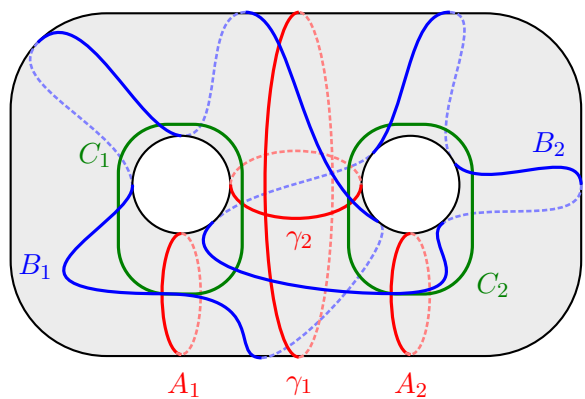

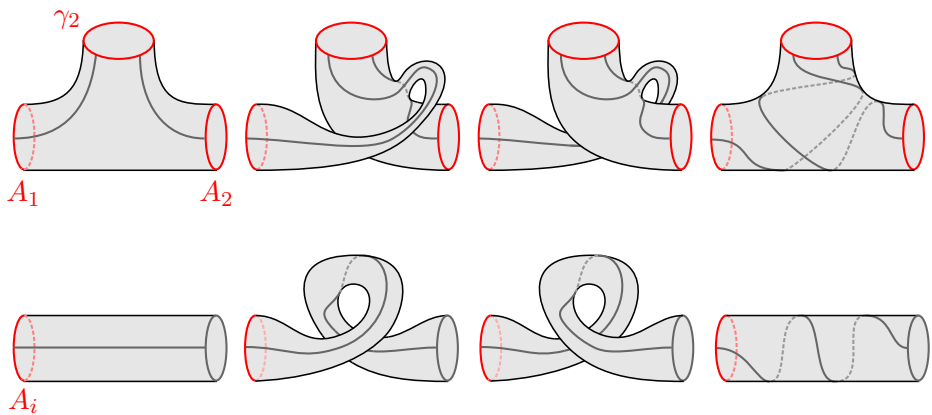
\begin{figure}
    \centering
    \labellist
    \pinlabel{\textcolor{red}{$A_1$}} at 3 52
    \pinlabel{\textcolor{red}{$A_2$}} at 56 52
    \pinlabel{\textcolor{red}{$\gamma_2$}} at 15 100
    \pinlabel{\textcolor{red}{$A_i$}} at 3 -5
    \endlabellist
    \includegraphics[width=0.8\textwidth]{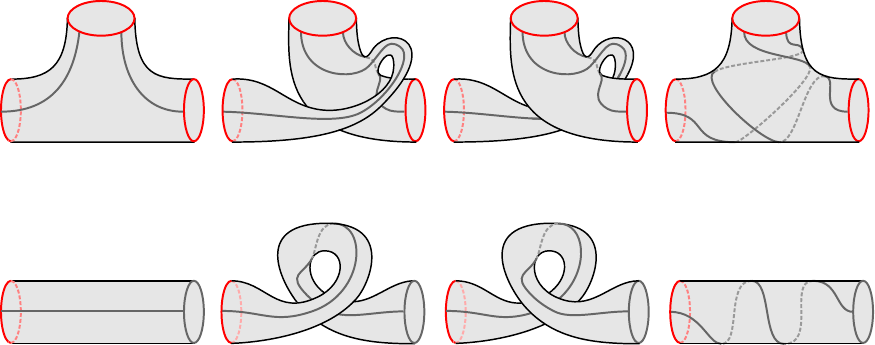}
    \vspace{.1in}
    \caption{Each row depicts an isotopy of $S^4$ taking $H_1$ to $H_1$ setwise, as in Remark \ref{fixh1remark}. In the top row, from left to right the induced automorphism on $U$ is isotopic to a product of a right-handed Dehn twists $A_1$ and $A_2$ and a left-handed Dehn twist about $\gamma_2$. In the bottom row, from left to right the induced automorphism of $U$ is isotopic to two right-handed Dehn twists about $A_i$. The handedness of all relevant Dehn twists can be reversed by reversing the illustrated isotopy.}
    \label{fig:h1isotopies}
\end{figure}

Note that $\gamma_1$ is separating in $U$. Then we may perform smooth isotopy of $S^4$ (extended to $S^5$) taking $H_1$ to itself (setwise) so that the induced automorphism on $U$ is a Dehn twist (of either sign) about $\gamma_1$. In the top row of Figure \ref{fig:h1isotopies}, we show how to perform another smooth isotopy of $S^4$ (extended to $S^5$) taking $H_1$ to itself so that the induced automorphism on $U$ is a composition of a Dehn twist about $\gamma_2$ (of either sign) and Dehn twists about $A_1$ and $A_2$ of the opposite sign. Thus, by performing a sequence of these isotopies before choosing $H_3$, we may assume that $\phi(\gamma_1)=\gamma_1$.

Now we have arranged so that $C_i=\phi(B_i)$ is obtained from $B_i$ by Dehn twists about $A_i$. Since $q([A_i])=q([B_i])=q([C_i])=0$, $C_i$ is obtained from $B_i$ by an even number of Dehn twists about $A_i$ for each $i$. In the bottom row of Figure \ref{fig:h1isotopies}, we show another smooth isotopy of $S^4$ (extended to $S^5$) taking $H_1$ to itself so that the induced automorphism on $U$ is given by two Dehn twists about $A_i$ (of either sign). By performing some number of these isotopies (again before choosing $H_3$) we may take $C_i=B_i$, so $H_2$ and $H_3$ are compressing curve equivalent. Then set $H_2^*:=H_3$ and push the interiors of both $H_2,H_2^*$ slightly into $V_2$. The smooth handlebodies $H_2,H_2^*$ are not topologically isotopic rel.\ boundary in $V_2\cong B^5$.

So far, we have only considered the genus-2 case. As in the proof of Theorem \ref{mainthm}, if we simultaneously add $g-2$ solid tubes to $H_2,H_2^*$, the resulting smooth genus-$g$ handlebodies are also not topologically isotopic rel.\ boundary in $V_2\cong B^5$.

\end{remark}

\bibliographystyle{plain}
\bibliography{biblio.bib}

\end{document}